\documentclass[11pt]{amsart}

\usepackage{amsmath, amssymb, amsthm, mathtools}
\usepackage{geometry}
\usepackage{tikz}
\usepackage[colorlinks=true, linkcolor=blue, citecolor=blue, urlcolor=blue]{hyperref}
\usepackage[capitalize, noabbrev]{cleveref}
\theoremstyle{plain}
\newtheorem{theorem}{Theorem}[section]
\newtheorem{lemma}[theorem]{Lemma}
\newtheorem{proposition}[theorem]{Proposition}
\newtheorem{corollary}[theorem]{Corollary}

\theoremstyle{definition}
\newtheorem{definition}[theorem]{Definition}

\theoremstyle{remark}
\newtheorem{remark}[theorem]{Remark}
\newtheorem{convention}[theorem]{Convention}

\Crefname{theorem}{Theorem}{Theorems}
\Crefname{lemma}{Lemma}{Lemmas}
\Crefname{proposition}{Proposition}{Propositions}
\Crefname{corollary}{Corollary}{Corollaries}
\Crefname{definition}{Definition}{Definitions}
\Crefname{remark}{Remark}{Remarks}
\Crefname{example}{Example}{Examples}
\Crefname{section}{\S}{\S\S}
\Crefname{subsection}{\S}{\S\S}

\DeclareMathOperator{\arsinh}{arsinh}
\DeclareMathOperator{\arcosh}{arcosh}
\DeclareMathOperator{\artanh}{artanh}
\DeclareMathOperator{\sech}{sech}
\DeclareMathOperator{\csch}{csch}

\newcommand{\R}{\mathbb{R}}

\newcommand{\Z}{\mathbb{Z}}

\newcommand{\eps}{\varepsilon}
\newcommand{\dd}{\,d}
\newcommand{\deriv}[1]{\frac{d}{d#1}}

\title{Integral-Differential Calculus}
\author{Grant Molnar}
\email{molnar.grant.5772@gmail.com}
\date{\today}

\begin{document}

\begin{abstract}
We give an exposition of the Newton-Leibniz calculus. We begin by defining the integral as a limit of Riemann sums, verify the integrals of the standard catalog of functions by direct manipulation, prove the substitution lemmas as theorems about Riemann sums, cross the Fundamental Theorem of Calculus, and harvest the differential calculus on the other side.
\end{abstract}

\maketitle

\section{Introduction}

In this note, we define the integral as a limit of Riemann sums, verify its standard properties, compute the integrals of the usual catalog of functions directly from the definition, prove $u$-substitution and integration by parts as theorems about Riemann sums, and then cross the Fundamental Theorem of Calculus and harvest the derivative table on the other side. 

Developing the Newton-Leibniz calculus this way is natural and beautiful. Several of the direct integral evaluations --- the geometric-series arguments for $\int b^t\dd t$, $\int_0^x \cos(t)\dd t$, and $\int_0^x \sin(t)\dd t$, the telescoping argument for $\int_0^x \sec^2(t)\dd t$ and $\int_a^b \csc^2(t)\dd t$, and the functional-equation derivation of $\log$ from $\int_1^x dt/t$ --- are instructive in their own right, and the chain rule and product rule fall out from $u$-substitution and integration by parts as pleasant corollaries.

Some authors have advocated for an alternative order, in which one begins by defining the derivative as a limit of difference quotients, verifies the analogous rules (linearity, product rule, chain rule), and derives the integral table from the derivative table after crossing the Fundamental Theorem of Calculus. This alternative has some computational advantages. Elementary derivatives are mechanical to compute, and a student following this order acquires a great deal of fluency with the formalism before encountering any difficult integration step. However, many students under this pedagogy never develop a firm geometric intuition for integration, and all lose the direct view of where the integral table comes from. Readers interested in the differential-first approach may consult the peculiar references on the subject, among which \cite{spivak} is representative.

\subsection*{Conventions and prerequisites}

We use the Riemann integral throughout. The gauge (Henstock--Kurzweil) integral \cite{bartle} has well-known advantages for proving the Fundamental Theorem of Calculus in maximum generality, but we will only ever integrate continuous functions, where the two integrals agree, and the Riemann construction is more familiar.

We assume the reader is comfortable with the construction of the real numbers and other basic ideas from real analysis such as completeness, continuity (including uniform and Lipschitz continuity), and the intermediate and extreme value theorems on a compact interval. We assume that sine and cosine have been defined geometrically (for instance via arc length on the unit circle, or as the unique continuous solutions to the addition formulas with appropriate normalizations) and that the addition and half-angle identities are at our disposal. We also recall the geometric observation that
\begin{equation}\label{eq:sin-over-theta}
    \lim_{\theta \to 0} \frac{\sin(\theta)}{\theta} = 1.
\end{equation}
No further prerequisites are presumed.

\section{The integral as primitive}\label{sec:integral}

We collect the definition of the Riemann integral and a handful of its standard properties.

\begin{definition}[Riemann integral]
A \emph{tagged partition} of $[a,b]$ is a sequence $a = t_0 < t_1 < \cdots < t_n = b$ together with tags $\xi_k \in [t_k, t_{k+1}]$. Its \emph{mesh} is $\max_k (t_{k+1} - t_k)$. The \emph{Riemann sum} of a function $f \colon [a,b] \to \R$ is
\[
S(f, P) \coloneqq \sum_{k=0}^{n-1} f(\xi_k)(t_{k+1} - t_k).
\]
We say $f$ is \emph{Riemann integrable} on $[a,b]$ with integral $I$ if for every $\eps > 0$ there is $\delta > 0$ such that $|S(f, P) - I| < \eps$ for every tagged partition $P$ of mesh less than $\delta$. We then write $\int_a^b f(t)\dd t = I$.
\end{definition}

We extend the integral to $a > b$ by the convention $\int_a^b f(t)\dd t = -\int_b^a f(t)\dd t$, and define $\int_a^a f(t)\dd t = 0$.

The basic properties of the Riemann integral are entirely standard and we recall them without proof; a treatment can be found in \cite{rudin} or any introductory real-analysis text.

\begin{proposition}\label{prop:standard}
If $f$ is continuous function on $[a,b]$, then $f$ is Riemann integrable on $[a,b]$. Now let $a, b, c, \alpha, \beta \in \R$, and let $f$ and $g$ be integrable functions. The Riemann integral satisfies the following properties.
\begin{enumerate}
\item \textup{(Linearity.)} We have
\[
\int_a^b \bigl(\alpha f(t) + \beta g(t)\bigr)\dd t = \alpha \int_a^b f(t)\dd t + \beta \int_a^b g(t)\dd t.
\]
\item \textup{(Additivity.)} We have
\[
\int_a^c f(t)\dd t = \int_a^b f(t)\dd t + \int_b^c f(t)\dd t.
\]
\item \textup{(Monotonicity.)} If $f \le g$ pointwise on $[a,b]$ and both are integrable, then 
\[
\int_a^b f(t)\dd t \le \int_a^b g(t)\dd t.
\]
In particular, 
\[
\Bigl|\int_a^b f(t)\dd t\Bigr| \le \int_a^b |f(t)|\dd t.
\]
\item \textup{(Continuity in the upper limit.)} If $f$ is bounded by $M$, the function $F(x) = \int_a^x f(t)\dd t$ is Lipschitz continuous on $[a,b]$ with constant $M$.
\end{enumerate}
\end{proposition}

We will repeatedly need two further consequences of continuity. The first is the integral version of the mean value theorem.

\begin{lemma}[Mean value theorem for integrals]\label{lem:MVT-int}
If $f$ is continuous on $[a,b]$, there exists $\xi \in [a,b]$ with
\[
\int_a^b f(t)\dd t = f(\xi)\,(b-a).
\]
\end{lemma}

\begin{proof}
Let $m = \min_{[a,b]} f$ and $M = \max_{[a,b]} f$. Monotonicity gives
\[
m(b-a) \le \int_a^b f(t)\dd t \le M(b-a),
\]
so the average value $\frac{1}{b-a}\int_a^b f(t)\dd t$ lies in $[m, M]$. By the intermediate value theorem, $f$ attains this value at some point $\xi \in [a, b]$.
\end{proof}

The second is a tag-perturbation lemma that says we may swap one tag for another inside a Riemann sum at vanishing cost.

\begin{lemma}[Riemann sums tolerate small perturbation of tags]\label{lem:tag-perturb}
Let $f$ be continuous on $[a,b]$ and $\eps > 0$. There is $\delta > 0$ such that, for any partition $a = t_0 < \cdots < t_n = b$ of mesh less than $\delta$ and any choices of tags $\xi_k, \xi_k' \in [t_k, t_{k+1}]$,
\[
\left| \sum_{k=0}^{n-1} \bigl(f(\xi_k) - f(\xi_k')\bigr)(t_{k+1} - t_k) \right| < \eps(b-a).
\]
\end{lemma}

\begin{proof}
By uniform continuity of $f$ on the compact interval $[a,b]$, choose $\delta$ so that $|s - t| < \delta$ implies $|f(s) - f(t)| < \eps$. Then $|f(\xi_k) - f(\xi_k')| < \eps$ for every $k$ and the bound follows by the triangle inequality.
\end{proof}

\begin{convention}
Throughout, all integrands are assumed continuous on the intervals over which we integrate. Where we write a one-sided integral such as $\int_0^1 dt/\sqrt{1-t^2}$ that has a singularity at an endpoint, we mean the improper integral, computed as a limit as the endpoint is approached.
\end{convention}

\section{Direct verification of classical integrals}\label{sec:direct}

We come to the main work of the integral side, the computation by hand and from first principles of the integrals
\[
\int \frac{\dd t}{t}, \quad \int e^t\dd t, \quad \int b^t\dd t, \quad \int t^n\dd t, \quad \int \cos(t)\dd t, \quad \int \sin(t)\dd t, \quad \int \sec^2(t)\dd t, \quad \int \csc^2(t)\dd t.
\]
The integral-first program has substantial precedent; a modern textbook treatment along these lines is given in Apostol \cite{apostol}, and several of the direct evaluations below date back to the seventeenth century, most notably Fermat's geometric-partition evaluation of $\int t^n\dd t$, for which see \cite{edwards}. Each verification is a direct manipulation of Riemann sums or a telescoping identity. None invokes a derivative or an antiderivative. The arguments fall into two general flavors. The first, which handles $1/t$, $b^t$, and the trigonometric pair $\sin(t)$ and $\cos(t)$, is a geometric-series argument. The second, which handles $\sec^2(t)$ and $\csc^2(t)$, is telescoping with a known multiplier. The polynomial integrand $t^n$ submits to a separate Faulhaber-style summation, or a functional equation argument when $n = -1$.

\subsection{The reciprocal integrand and the natural logarithm}\label{sec:log}

We begin by defining the natural logarithm.

\begin{definition}\label{def:log}
For $x > 0$, the \emph{natural logarithm} is
\[
\log(x) \coloneqq \int_1^x \frac{\dd t}{t}.
\]
\end{definition}

This is unmotivated for the moment, in the sense that it is not obvious from the definition that $\log$ has anything to do with logarithms in the usual sense. The next several results justify the name. First, an explicit limit form.

\begin{lemma}\label{lem:log-geometric}
For every $x > 0$,
\[
\log(x) = \lim_{n \to \infty} n\bigl(x^{1/n} - 1\bigr).
\]
\end{lemma}

\begin{proof}
Suppose first $x > 1$. Partition $[1, x]$ geometrically by setting $u_k = x^{k/n}$ for $k = 0, 1, \ldots, n$. The widths are $u_{k+1} - u_k = x^{k/n}(x^{1/n} - 1)$, and on the cell $[u_k, u_{k+1}]$ the integrand $1/u$ satisfies $1/u_{k+1} \le 1/u \le 1/u_k$. Hence
\[
\sum_{k=0}^{n-1} \frac{x^{k/n}(x^{1/n}-1)}{u_{k+1}}
\;\le\; \int_1^x \frac{\dd u}{u} \;\le\;
\sum_{k=0}^{n-1} \frac{x^{k/n}(x^{1/n}-1)}{u_k}.
\]
The right sum simplifies term by term to $x^{1/n} - 1$, totaling $n(x^{1/n} - 1)$. The left sum simplifies to $(x^{1/n}-1)/x^{1/n}$ per term, totaling $n(x^{1/n}-1)/x^{1/n}$. Bernoulli's inequality gives $x = (1 + (x^{1/n} - 1))^n \ge 1 + n(x^{1/n} - 1)$, hence $x^{1/n} - 1 \le (x-1)/n$, which forces $x^{1/n} \to 1$. The gap between the two bounds is
\[
n(x^{1/n}-1) - \frac{n(x^{1/n}-1)}{x^{1/n}} = n(x^{1/n}-1) \cdot \frac{x^{1/n}-1}{x^{1/n}} \le \frac{(x-1)^2}{n \cdot x^{1/n}},
\]
which tends to $0$ since $x^{1/n} \to 1$. Since the integral $\log(x) = \int_1^x \dd u/u$ is sandwiched between the two bounds and the gap vanishes, both bounds converge to $\log(x)$; in particular $n(x^{1/n} - 1) \to \log(x)$.

For $0 < x < 1$, run the same partition in reverse on $[x, 1]$; the bounds are analogous and the conclusion follows from the convention $\int_1^x = -\int_x^1$.
\end{proof}

The point of calling this function the logarithm is the following functional equation, which Cauchy \cite{cauchy-cours} proved is the defining property of any logarithm; for a modern exposition see \cite{aczel}.

\begin{theorem}[Functional equation]\label{thm:log-fe}
For all $x, y > 0$,
\[
\log(xy) = \log(x) + \log(y).
\]
\end{theorem}

\begin{proof}
By additivity of the integral, 
\[
\log(xy) = \int_1^{xy} dt/t = \int_1^x dt/t + \int_x^{xy} dt/t = \log(x) + \int_x^{xy} dt/t.
\]
It remains to show that the second piece equals $\log(y)$.

Partition $[x, xy]$ geometrically by $v_k = x \cdot y^{k/n}$ for $k = 0, \ldots, n$, assuming $y > 1$ (the case $y < 1$ is symmetric). Each cell width is $xy^{k/n}(y^{1/n} - 1)$. The right-endpoint Riemann sum for $\int_x^{xy} dt/t$ is
\[
\sum_{k=0}^{n-1} \frac{xy^{k/n}(y^{1/n}-1)}{xy^{(k+1)/n}} = \sum_{k=0}^{n-1} \frac{y^{1/n}-1}{y^{1/n}} = \frac{n(y^{1/n}-1)}{y^{1/n}},
\]
which tends to $\log(y)$ by \Cref{lem:log-geometric}. The left-endpoint Riemann sum gives $n(y^{1/n}-1)$ by the same calculation, also tending to $\log(y)$. The integrand $1/u$ is continuous and the partition mesh tends to $0$, so the Riemann sums converge to the integral, which therefore equals $\log(y)$.
\end{proof}

The crucial observation is that $\int_x^{xy} dt/t$ depends only on $y$, not on $x$. Geometrically, dilating the interval $[1, y]$ by a factor of $x$ leaves the integral of $1/t$ unchanged, because the dilation widens each subinterval by a factor of $x$ and shrinks the integrand by the same factor. This scale-invariance is the source of the logarithmic functional equation.

\begin{corollary}\label{cor:log-properties}
The function $\log\colon (0, \infty) \to \R$ is continuous, strictly increasing, satisfies $\log(1) = 0$ and $\log(1/x) = -\log(x)$, satisfies $\log(x^n) = n\log(x)$ for all $n \in \Z$, and is a bijection onto $\R$.
\end{corollary}

\begin{proof}
Continuity is \Cref{prop:standard}(5). Monotonicity comes from positivity of the integrand, since for $a < b$, $\log(b) - \log(a) = \int_a^b dt/t > 0$. From the functional equation, $\log(x \cdot 1) = \log(x) + \log(1)$, so $\log(1) = 0$, and
\[
\log(x) + \log(1/x) = \log(1) = 0
\]
gives $\log(1/x) = -\log(x)$. Induction on the functional equation gives $\log(x^n) = n\log(x)$ for $n \ge 0$, and the result for $n < 0$ follows by inverting. Surjectivity comes from $\log(2) > 0$ (positive integrand on $[1, 2]$), since then $\log(2^n) = n\log(2) \to \infty$ and $\log(2^{-n}) \to -\infty$, and the intermediate value theorem fills in the rest.
\end{proof}

Having established that $\log$ is a continuous bijection from $(0, \infty)$ to $\R$, we may define its inverse.

\begin{definition}\label{def:exp}
The \emph{exponential function} $\exp \colon \R \to (0, \infty)$ is the inverse of $\log$. We set $e \coloneqq \exp(1)$, characterized by $\log(e) = 1$. For any $b > 0$, we define
\[
b^x = \exp(x \log(b)).
\]
\end{definition}

The third part of the definition extends the operation of raising a positive real to a real power. For integer or rational exponents the definition agrees with the elementary one. For $n \in \Z$ we have $\log(b^n) = n\log(b)$ from the functional equation, so $\exp(n \log(b)) = b^n$. For rational $p/q$ with $q > 0$, the relation $(b^{1/q})^q = b$ gives $q \log(b^{1/q}) = \log(b)$, hence $b^{1/q} = \exp(\log(b)/q)$, and the rest of the rationals follow by combining these two arguments. For irrational exponents, the definition is the natural continuous extension.

The next lemma is the analytic ingredient that makes the exponential integral come out right.

\begin{lemma}\label{lem:exp-derivative-at-0}
$\displaystyle \lim_{h \to 0} \frac{e^h - 1}{h} = 1$.
\end{lemma}

\begin{proof}
Set $u = e^h - 1$. As $h \to 0$, $u \to 0$ as well, by continuity of $\exp$ at $0$ (and $\exp$ is continuous because it is the inverse of the continuous strictly monotone $\log$). Now
\[
\log(1 + u) = \log(e^h) = h.
\]
By \Cref{lem:MVT-int} applied to $1/t$ on the interval between $1$ and $1 + u$, we have
\[
h = \log(1 + u) = u/\xi
\]
for some $\xi$ between $1$ and $1 + u$. Hence $h/u = 1/\xi$, which tends to $1$ as $u \to 0$ by continuity. Inverting, we conclude $u/h \to 1$.
\end{proof}

\begin{theorem}\label{thm:int-exp}
For all $p, q \in \R$ and $b > 0$ with $b \neq 1$, we have
\[
\int_p^q b^t\dd t = \frac{b^q - b^p}{\log(b)}.
\]
In particular, we have
\[
\int_p^q e^t\dd t = e^q - e^p,
\]
\end{theorem}

\begin{proof}
We begin with the second identity. Partition $[p, q]$ uniformly with $t_k = p + k\Delta$ where $\Delta = (q - p)/n$. The Riemann sum is the geometric series
\[
\Delta \sum_{k=0}^{n-1} e^{p + k\Delta}
\;=\; \Delta \cdot e^p \cdot \frac{e^{n\Delta} - 1}{e^\Delta - 1}
\;=\; (e^q - e^p)\cdot \frac{\Delta}{e^\Delta - 1}.
\]
As $n \to \infty$, $\Delta \to 0$, and \Cref{lem:exp-derivative-at-0} gives $\Delta/(e^\Delta - 1) \to 1$. The integral therefore equals $e^q - e^p$.

For the first identity, the same uniform partition gives the analogous Riemann sum
\[
\Delta \sum_{k=0}^{n-1} b^{p + k\Delta}
\;=\; (b^q - b^p) \cdot \frac{\Delta}{b^\Delta - 1}.
\]
Now $b^\Delta = e^{\Delta \log(b)}$ from \Cref{def:exp}, so writing $h = \Delta \log(b)$,
\[
\frac{\Delta}{b^\Delta - 1} = \frac{1}{\log(b)} \cdot \frac{h}{e^h - 1} \;\longrightarrow\; \frac{1}{\log(b)}
\]
as $\Delta \to 0$ (and hence $h \to 0$), again by \Cref{lem:exp-derivative-at-0}. The integral therefore equals $(b^q - b^p)/\log(b)$.
\end{proof}

\subsection{Polynomial integrands}

\begin{theorem}\label{thm:int-power}
For every nonnegative integer $n$ and every $x \in \R$, we have
\[
\int_0^x t^n\dd t = \frac{x^{n+1}}{n+1}.
\]
\end{theorem}

\begin{proof}
Take $x > 0$; the negative case is symmetric. With uniform partition $t_k = kx/N$ for $k = 0, \ldots, N$, the left-endpoint Riemann sum is
\[
S_N = \frac{x}{N} \sum_{k=0}^{N-1} \left(\frac{kx}{N}\right)^n = \frac{x^{n+1}}{N^{n+1}} \sum_{k=0}^{N-1} k^n.
\]
We claim that $\sum_{k=0}^{N-1} k^n = N^{n+1}/(n+1) + O(N^n)$, which is the leading term of Faulhaber's formula \cite{knuth-faulhaber}. To see this, telescope the identity $(k+1)^{n+1} - k^{n+1} = \sum_{j=0}^n \binom{n+1}{j} k^j$ from $k = 0$ to $N - 1$ to get
\[
N^{n+1} = \sum_{j=0}^n \binom{n+1}{j} \sum_{k=0}^{N-1} k^j = (n+1) \sum_{k=0}^{N-1} k^n + \sum_{j=0}^{n-1} \binom{n+1}{j} \sum_{k=0}^{N-1} k^j.
\]
For each $j < n$, the inner sum $\sum_{k=0}^{N-1} k^j$ is at most $N \cdot N^j = N^{j+1} \le N^n$, so the second term on the right is $O(N^n)$. Solving, $\sum_{k=0}^{N-1} k^n = N^{n+1}/(n+1) + O(N^n)$. Substituting,
\[
S_N = \frac{x^{n+1}}{n+1} + O(1/N) \;\longrightarrow\; \frac{x^{n+1}}{n+1}. \qedhere
\]
\end{proof}

\subsection{Sine and cosine}\label{sec:int-sin-cos}

The combination $\cos(\theta) + i \sin(\theta)$ is the source of much trigonometric magic. We need de Moivre's formula, which is itself the inductive consequence of the angle-addition identities.

\begin{lemma}[de Moivre]\label{lem:demoivre}
For every $\theta \in \R$ and every nonnegative integer $n$, we have
\[
(\cos(\theta) + i\sin(\theta))^n = \cos(n\theta) + i\sin(n\theta).
\]
\end{lemma}

\begin{proof}
The claim follows by induction on $n$. The base case $n = 0$ is the trivial identity $1 = 1$. For the inductive step,
\begin{align*}
(\cos(\theta) + i\sin(\theta))^{n+1}
&= (\cos(n\theta) + i\sin(n\theta))(\cos(\theta) + i\sin(\theta)) \\
&= \bigl(\cos(n\theta)\cos(\theta) - \sin(n\theta)\sin(\theta)\bigr) + i\bigl(\sin(n\theta)\cos(\theta) + \cos(n\theta)\sin(\theta)\bigr) \\
&= \cos((n+1)\theta) + i\sin((n+1)\theta),
\end{align*}
using the angle-addition identities at the last step.
\end{proof}

\begin{theorem}\label{thm:int-cos-sin}
For every $x \in \R$,
\[
\int_0^x \cos(t)\dd t = \sin(x),
\qquad
\int_0^x \sin(t)\dd t = 1 - \cos(x).
\]
\end{theorem}

\begin{proof}
The two identities are the real and imaginary parts of a single complex statement, so we compute the complex Riemann sum for $\int_0^x (\cos(t) + i\sin(t))\dd t$ and read off both. Use the uniform partition $t_k = kx/n$, and set $z = \cos(x/n) + i\sin(x/n)$. By \Cref{lem:demoivre}, $z^k = \cos(t_k) + i\sin(t_k)$. The Riemann sum is then a finite geometric series,
\[
S_n = \frac{x}{n}\sum_{k=0}^{n-1} z^k = \frac{x}{n}\cdot\frac{z^n - 1}{z - 1}.
\]

To pass to the limit, we expose the structure of $z^n - 1$ and $z - 1$ using the half-angle identities $\cos(\theta) - 1 = -2\sin^2(\theta/2)$ and $\sin(\theta) = 2\sin(\theta/2)\cos(\theta/2)$. Direct computation gives
\[
z^n - 1 = \cos(x) - 1 + i\sin(x) = -2\sin^2(x/2) + 2i\sin(x/2)\cos(x/2) = 2i\sin(x/2)\,W,
\]
where $W = \cos(x/2) + i\sin(x/2)$. Likewise,
\[
z - 1 = 2i\sin(x/(2n))\,w, \qquad w = \cos(x/(2n)) + i\sin(x/(2n)).
\]
Substituting,
\[
S_n = \frac{x}{n} \cdot \frac{\sin(x/2)}{\sin(x/(2n))} \cdot \frac{W}{w}
= 2 \cdot \frac{x/(2n)}{\sin(x/(2n))} \cdot \sin(x/2) \cdot \frac{W}{w}.
\]

Three factors approach known limits as $n \to \infty$. The first, $(x/(2n))/\sin(x/(2n))$, tends to $1$ by \eqref{eq:sin-over-theta}. The second, $\sin(x/2)$, is constant in $n$. The third, $w$, tends to $1$ by continuity of $\cos$ and $\sin$ at $0$, while $W$ is constant. Therefore
\[
\lim_n S_n = 2\sin(x/2)\,W = 2\sin(x/2)\cos(x/2) + 2i\sin^2(x/2) = \sin(x) + i\,(1 - \cos(x)),
\]
using the half-angle identities one more time. The real and imaginary parts of $S_n$ are the Riemann sums for $\int_0^x \cos(t)\dd t$ and $\int_0^x \sin(t)\dd t$ respectively, and they converge to $\sin(x)$ and $1 - \cos(x)$ as claimed.
\end{proof}

\subsection{Squared secant and squared cosecant}\label{sec:sec2-csc2}

The integrals of $\sec^2(t)$ and $\csc^2(t)$ submit to a telescoping argument that uses only the angle-subtraction formulas. The flavor is parallel to the de Moivre / geometric-series gambit just used. An addition or subtraction formula for the would-be antiderivative gives a multiplicative factor times a sample of the integrand, with the factor having a known limit as the partition refines.

\begin{lemma}\label{lem:tan-diff}
For all $A, B \in \R$ with $\cos(A), \cos(B) \neq 0$, we have
\[
\tan(A) - \tan(B) = \frac{\sin(A - B)}{\cos(A) \cos(B)}.
\]
Similarly, for all $A, B \in \R$ with $\sin(A), \sin(B) \neq 0$, we have
\[
\cot(B) - \cot(A) = \frac{\sin(A - B)}{\sin(A)\sin(B)}.
\]
\end{lemma}

\begin{proof}
Direct algebra and the sine sum-angle identity gives 
\[
\tan(A) - \tan(B) = \frac{\sin(A)\cos(B) - \cos(A)\sin(B)}{\cos(A)\cos(B)} = \frac{\sin(A - B)}{\cos(A)\cos(B)}.
\]
The cotangent version is symmetric.
\end{proof}

\begin{theorem}\label{thm:int-sec2}
For $x \in (-\pi/2, \pi/2)$, we have
\[
\int_0^x \sec^2(t)\dd t = \tan(x).
\]
For $0 < a < b < \pi$, we have
\[
\int_a^b \csc^2(t)\dd t = \cot(a) - \cot(b).
\]
\end{theorem}

\begin{proof}
For the first identity, partition $[0, x]$ uniformly at $t_k = kx/n$ and telescope using \Cref{lem:tan-diff}, to obtain
\[
\tan(x) = \sum_{k=0}^{n-1} \bigl[\tan(t_{k+1}) - \tan(t_k)\bigr]
= \sum_{k=0}^{n-1} \frac{\sin(x/n)}{\cos(t_{k+1})\cos(t_k)}.
\]
Pulling out the sine factor and rewriting as a Riemann-sum-with-correction, we see
\[
\tan(x) = \frac{\sin(x/n)}{x/n} \cdot \frac{x}{n} \sum_{k=0}^{n-1} \frac{1}{\cos(t_{k+1})\cos(t_k)}.
\]
The leading factor tends to $1$ by \eqref{eq:sin-over-theta}. For the sum, $\cos(t)$ is continuous on the compact interval $[0, x]$ and bounded away from zero, so $\cos(t_{k+1}) = \cos(t_k) + O(1/n)$ uniformly in $k$ by uniform continuity. Hence
\[
\frac{1}{\cos(t_{k+1})\cos(t_k)} = \sec^2(t_k) + \eta_{k,n}
\]
with $\max_k |\eta_{k,n}| \to 0$. The rightmost factor therefore differs from a Riemann sum for $\int_0^x \sec^2(t)\dd t$ by at most $\max_k|\eta_{k,n}| \cdot x$, which vanishes in the limit. Taking $n \to \infty$ yields the integral identity.

The cosecant identity is identical with the cotangent version of \Cref{lem:tan-diff}. The requirement $0 < a$ encodes the (true and necessary) fact that $\csc^2(t)$ has a nonintegrable singularity at $0$.
\end{proof}

\begin{remark}
Our exponential and trigonometric identities all run on the same engine. An addition or subtraction formula expresses the increment of the would-be antiderivative as a multiplicative factor (depending only on the partition step) times a sample of the integrand. The factor has a known limit as the step shrinks. One might call it ``telescoping with a known multiplier.'' Once the reader sees this twice, they will start spotting it everywhere.
\end{remark}

\section{The substitution lemmas}\label{sec:sub}

We have built a respectable starter table of integrals. To extend it to the rest of the standard catalog --- $\int t^a \dd t$ for general real $a$, the inverse trigonometric integrals, the missing $\int \tan(t) \dd t$ and $\int \cot(t) \dd t$, and the rest --- we need substitution theorems. The two substitution lemmas in the standard calculus are $u$-substitution and integration by parts. We now prove these both.

\subsection{Change of variables}\label{sec:u-sub}

\begin{theorem}[Change of variables, integral form]\label{thm:u-sub}
Let $g \colon [a,b] \to \R$ be continuous and let $G \colon [a, b] \to \R$ be a continuous function satisfying
\[
G(x) - G(a) = \int_a^x g(t)\dd t \quad \text{for every } x \in [a, b].
\]
Let $f$ be continuous on a closed interval containing $G([a,b])$. Then
\[
\int_a^b f(G(t))\,g(t)\dd t = \int_{G(a)}^{G(b)} f(u)\dd u,
\]
with the convention $\int_\beta^\alpha = -\int_\alpha^\beta$ when $\beta < \alpha$.
\end{theorem}

\begin{proof}
Define $\Phi \colon [a, b] \to \R$ by $\Phi(x) = \int_{G(a)}^{G(x)} f(u)\dd u$, with the sign convention. We will show that $\Phi(x) = \int_a^x f(G(t))\,g(t)\dd t$ for every $x \in [a, b]$, and the case $x = b$ is the theorem.

Take any partition $a = s_0 < s_1 < \cdots < s_N = x$. Telescope $\Phi$ along the partition and rewrite each piece as an integral by \Cref{prop:standard}(2),
\[
\Phi(x) = \sum_{k=0}^{N-1} \bigl[\Phi(s_{k+1}) - \Phi(s_k)\bigr] = \sum_{k=0}^{N-1} \int_{G(s_k)}^{G(s_{k+1})} f(u)\dd u.
\]
By \Cref{lem:MVT-int} applied to $f$ on the interval between $G(s_k)$ and $G(s_{k+1})$ (the formula is valid even when $G(s_{k+1}) < G(s_k)$, since negating both sides recovers the standard direction), each summand equals $f(\eta_k)\,(G(s_{k+1}) - G(s_k))$ for some $\eta_k$ between $G(s_k)$ and $G(s_{k+1})$. By the intermediate value theorem applied to the continuous $G$ on $[s_k, s_{k+1}]$, we may write $\eta_k = G(s_k')$ for some $s_k' \in [s_k, s_{k+1}]$.

Next, by \Cref{lem:MVT-int} applied to $g$ on $[s_k, s_{k+1}]$, the increment $G(s_{k+1}) - G(s_k) = \int_{s_k}^{s_{k+1}} g(t)\dd t$ equals $g(\xi_k)\,(s_{k+1} - s_k)$ for some $\xi_k \in [s_k, s_{k+1}]$. Combining,
\[
\Phi(x) = \sum_{k=0}^{N-1} f(G(s_k'))\, g(\xi_k)\,(s_{k+1} - s_k).
\]

This is exactly a Riemann sum for $\int_a^x f(G(t))\,g(t)\dd t$, except that the tag inside $f \circ G$ is $s_k'$ and the tag inside $g$ is $\xi_k$, possibly distinct points of the same cell. The integrand $f(G(t))g(t)$ is continuous, and $|g| \le M$ on $[a, b]$ for some $M$. By \Cref{lem:tag-perturb} applied to $f \circ G$, with the $g$-factor of magnitude at most $M$ pulled aside, the difference between the displayed sum and a Riemann sum for $\int_a^x f(G(t))\,g(t)\dd t$ with a single common tag per cell is bounded by $\eps M (x - a)$ once the mesh is small enough. Hence $\Phi(x) = \lim_{\text{mesh}\to 0} (\text{Riemann sums}) = \int_a^x f(G(t))\,g(t)\dd t$.
\end{proof}

\subsection{Integration by parts}\label{sec:uv-sub}

\begin{theorem}[Integration by parts, integral form]\label{thm:uv-sub}
Let $u, v \colon [a, b] \to \R$ be continuous and let $p, q$ be continuous on $[a, b]$ with
\[
u(x) - u(a) = \int_a^x p(t)\dd t, \qquad v(x) - v(a) = \int_a^x q(t)\dd t
\]
for every $x \in [a, b]$. Then
\[
\int_a^b p(t)\, v(t)\dd t \;+\; \int_a^b u(t)\,q(t)\dd t \;=\; u(b)v(b) - u(a)v(a).
\]
\end{theorem}

\begin{proof}
Take any partition $a = t_0 < \cdots < t_n = b$. Telescope the difference of products,
\[
u(b)v(b) - u(a)v(a) = \sum_{k=0}^{n-1} \bigl[u(t_{k+1})v(t_{k+1}) - u(t_k)v(t_k)\bigr],
\]
then apply Abel's identity per term (see \cite{rudin} for the summation-by-parts form),
\[
u(t_{k+1})v(t_{k+1}) - u(t_k)v(t_k) = \bigl(u(t_{k+1}) - u(t_k)\bigr)v(t_{k+1}) + u(t_k)\bigl(v(t_{k+1}) - v(t_k)\bigr).
\]
By \Cref{lem:MVT-int} applied to $p$ and $q$ in turn,
\[
u(t_{k+1}) - u(t_k) = p(\xi_k)\,(t_{k+1}-t_k), \qquad
v(t_{k+1}) - v(t_k) = q(\eta_k)\,(t_{k+1}-t_k),
\]
for some $\xi_k, \eta_k \in [t_k, t_{k+1}]$. Substituting,
\[
u(b)v(b) - u(a)v(a) = \sum_k p(\xi_k)\,v(t_{k+1})\,(t_{k+1}-t_k) + \sum_k u(t_k)\,q(\eta_k)\,(t_{k+1}-t_k).
\]

Each sum on the right is a Riemann sum for one of the two integrals on the left side of the identity, modulo small mismatches in tags ($v(t_{k+1})$ vs.\ $v(\xi_k)$ in the first, and $u(t_k)$ vs.\ $u(\eta_k)$ in the second). Both $u$ and $v$ are continuous on $[a, b]$, so by \Cref{lem:tag-perturb} the discrepancies vanish in the limit as the mesh shrinks. Pass to the limit.
\end{proof}

\section{Derived integrals}\label{sec:derived}

With \Cref{thm:u-sub,thm:uv-sub} in hand, the rest of the standard integral table opens up.

\subsection{Real powers}

\begin{proposition}\label{prop:int-xa}
For any $a \in \R \setminus \{-1\}$ and any $0 < p < q$,
\[
\int_p^q t^a\dd t = \frac{q^{a+1} - p^{a+1}}{a+1}.
\]
\end{proposition}

\begin{proof}
Apply \Cref{thm:u-sub} with $G(s) = e^s$ and $g(s) = e^s$ on $[\log(p), \log(q)]$, where $G$ has the required form by \Cref{thm:int-exp}. Take $f(u) = u^a$ continuous on $[p, q]$. Then
\[
\int_p^q u^a\dd u = \int_{G(\log(p))}^{G(\log(q))} u^a\dd u = \int_{\log(p)}^{\log(q)} (e^s)^a \cdot e^s \dd s = \int_{\log(p)}^{\log(q)} e^{(a+1)s}\dd s.
\]
The last integral evaluates by \Cref{thm:int-exp} to
\[
\frac{e^{(a+1)\log(q)} - e^{(a+1)\log(p)}}{a+1} = \frac{q^{a+1} - p^{a+1}}{a+1}.
\]
\end{proof}

\subsection{The inverse trigonometric integrands}\label{sec:inverse-trig}

\begin{proposition}\label{prop:arctan}
For all $y \in \R$,
\[
\int_0^y \frac{\dd t}{1+t^2} = \arctan(y),
\]
where $\arctan$ is the inverse of $\tan$ restricted to $(-\pi/2, \pi/2)$.
\end{proposition}

\begin{proof}
Apply \Cref{thm:u-sub} with $G(s) = \tan(s)$ and $g(s) = \sec^2(s)$ on $[0, \arctan(y)]$, where $G$ has the required form by \Cref{thm:int-sec2}. Take $f(u) = 1/(1 + u^2)$ continuous on $[0, y]$. Using the Pythagorean identity $1 + \tan^2(s) = \sec^2(s)$,
\[
\int_0^y \frac{\dd u}{1 + u^2}
= \int_0^{\arctan(y)} \frac{\sec^2(s)}{1 + \tan^2(s)}\dd s
= \int_0^{\arctan(y)} 1\dd s
= \arctan(y). \qedhere
\]
\end{proof}

\begin{proposition}\label{prop:arcsin}
For $y \in (-1, 1)$,
\[
\int_0^y \frac{\dd t}{\sqrt{1-t^2}} = \arcsin(y).
\]
The improper-integral extension to $|y| = 1$ gives $\pi/2$.
\end{proposition}

\begin{proof}
Apply \Cref{thm:u-sub} with $G(s) = \sin(s)$ and $g(s) = \cos(s)$ on $[0, \arcsin(y)]$, where $G$ has the required form by \Cref{thm:int-cos-sin}, and take $f(u) = 1/\sqrt{1 - u^2}$. On $(-\pi/2, \pi/2)$, $\sqrt{1 - \sin^2(s)} = \cos(s) > 0$, so
\[
\int_0^y \frac{\dd u}{\sqrt{1 - u^2}}
= \int_0^{\arcsin(y)} \frac{\cos(s)}{\sqrt{1 - \sin^2(s)}}\dd s
= \int_0^{\arcsin(y)} 1\dd s
= \arcsin(y). \qedhere
\]
\end{proof}

\subsection{Tangent, cotangent, and the apologetic substitutions for secant and cosecant}

\begin{proposition}\label{prop:int-tan-cot}
On any interval where the integrand is continuous,
\[
\int \tan(t)\dd t = -\log|\cos(t)| + C, \qquad \int \cot(t)\dd t = \log|\sin(t)| + C.
\]
\end{proposition}

\begin{proof}
For $\tan(t)$, on an interval where $\cos(t)$ has constant sign, take $G(s) = \cos(s)$ with defining integrand $g(s) = -\sin(s)$ (by \Cref{thm:int-cos-sin} after a sign flip). With $f(u) = -1/u$, \Cref{thm:u-sub} gives
\[
\int_a^b \tan(t)\dd t = \int_a^b \frac{-(-\sin(t))}{\cos(t)}\dd t = \int_{\cos(a)}^{\cos(b)} \frac{-\dd u}{u} = -\bigl(\log|\cos(b)| - \log|\cos(a)|\bigr).
\]
The cotangent identity is symmetric.
\end{proof}

For $\sec(t)$ and $\csc(t)$, the standard trick is to multiply by $(\sec(t) + \tan(t))/(\sec(t) + \tan(t))$ and recognize the result as $g/G$ for $G(t) = \sec(t) + \tan(t)$. The substitution that pulls this off is a rabbit out of a hat in any framework, ours included. We present it without apology.

\begin{proposition}\label{prop:int-sec}
On $(-\pi/2, \pi/2)$,
\[
\int_0^x \sec(t)\dd t = \log|\sec(x) + \tan(x)|.
\]
\end{proposition}

\begin{proof}
Write
\[
\sec(t) = \sec(t) \cdot \frac{\sec(t) + \tan(t)}{\sec(t) + \tan(t)} = \frac{\sec^2(t) + \sec(t)\tan(t)}{\sec(t) + \tan(t)}.
\]
Set $G(t) = \sec(t) + \tan(t)$. Its defining integrand is $g(t) = \sec^2(t) + \sec(t)\tan(t)$, since $\int_0^x \sec^2(t)\dd t = \tan(x)$ by \Cref{thm:int-sec2} and $\int_0^x \sec(t)\tan(t)\dd t = \sec(x) - 1$ (a small telescoping verification analogous to \Cref{thm:int-sec2}, which we leave to the reader). With $f(u) = 1/u$, \Cref{thm:u-sub} delivers
\[
\int_0^x \sec(t)\dd t = \int_{G(0)}^{G(x)} \frac{\dd u}{u} = \log\bigl|G(x)\bigr| - \log\bigl|G(0)\bigr| = \log|\sec(x) + \tan(x)|. \qedhere
\]
\end{proof}

The cosecant case runs the same machinery with $G(t) = -(\csc(t) + \cot(t))$.

\subsection{The hyperbolic integrals}\label{sec:hyperbolic}

The hyperbolic functions are defined in terms of the exponential by 
\[
\cosh(t) \coloneqq \frac{e^t + e^{-t}}{2}, \quad \sinh(t) \coloneqq \frac{e^t - e^{-t}}{2},
\]
and the rest by the obvious quotients. Their integrals form a parallel catalog to the trigonometric one, derivable directly from \Cref{thm:int-exp} together with the same kind of telescoping argument used for $\sec^2(t)$.

\begin{proposition}\label{prop:hyperbolic}
For all $x \in \R$, we have
\[
\int_0^x \cosh(t)\dd t = \sinh(x), \qquad \int_0^x \sinh(t)\dd t = \cosh(x) - 1, \qquad \int_0^x \sech^2(t)\dd t = \tanh(x).
\]
For any $a, b$ with $0 < a < b$,
\[
\int_a^b \csch^2(t)\dd t = \coth(a) - \coth(b).
\]
For all $y \in \R$, $y \ge 1$, and $|y| < 1$ respectively, we have
\[
\int_0^y \frac{\dd t}{\sqrt{1 + t^2}} = \arsinh(y), \quad
\int_1^y \frac{\dd t}{\sqrt{t^2 - 1}} = \arcosh(y), \quad
\int_0^y \frac{\dd t}{1 - t^2} = \artanh(y).
\]
\end{proposition}

\begin{proof}[Proof sketch]
The first two identities are immediate from \Cref{thm:int-exp} and linearity, since $\cosh(t)$ and $\sinh(t)$ are linear combinations of $e^t$ and $e^{-t}$. The identities for $\sech^2(t)$ and $\csch^2(t)$ follow from the telescoping argument of \Cref{thm:int-sec2} applied to the hyperbolic analogue 
\[
\tanh(A) - \tanh(B) = \frac{\sinh(A - B)}{\cosh(A)\cosh(B)}
\]
and its cotangent counterpart. The three inverse hyperbolic identities follow by \Cref{thm:u-sub} applied to $G(t) = \sinh(t)$, $\cosh(t)$, and $\tanh(t)$ in turn, with the Pythagorean-style identity $\cosh^2(t) - \sinh^2(t) = 1$ playing the role of $\sin^2(t) + \cos^2(t) = 1$ in \Cref{prop:arcsin,prop:arctan}.
\end{proof}

\subsection{A showcase for integration by parts}

Integration by parts has been waiting in the wings. Its first natural application is the integral of the logarithm itself.

\begin{proposition}\label{prop:int-log}
For $x > 0$,
\[
\int_1^x \log(t)\dd t = x \log(x) - x + 1.
\]
\end{proposition}

\begin{proof}
Apply \Cref{thm:uv-sub} with $u(t) = \log(t)$ (so that $p(t) = 1/t$ by \Cref{def:log}) and $v(t) = t$ (so that $q(t) = 1$ by \Cref{thm:int-power} with $n = 0$). Then
\[
\int_1^x \log(t) \cdot 1\dd t + \int_1^x \frac{1}{t}\cdot t\dd t = x \log(x) - 1 \cdot \log(1).
\]
The right side simplifies because $\log(1) = 0$, and the second integral on the left is $\int_1^x 1\dd t = x - 1$. Solving,
\[
\int_1^x \log(t)\dd t = x\log(x) - (x - 1) = x\log(x) - x + 1.
\]
\end{proof}

A second application, using \Cref{thm:u-sub} as a sub-step, is the integral of arctangent.

\begin{proposition}\label{prop:int-arctan}
For $x \in \R$,
\[
\int_0^x \arctan(t)\dd t = x \arctan(x) - \tfrac{1}{2}\log(1 + x^2).
\]
\end{proposition}

\begin{proof}
Apply \Cref{thm:uv-sub} with $u(t) = \arctan(t)$ (so $p(t) = 1/(1+t^2)$ by \Cref{prop:arctan}) and $v(t) = t$,
\[
\int_0^x \arctan(t)\dd t + \int_0^x \frac{t}{1 + t^2}\dd t = x \arctan(x).
\]
For the second integral, take $G(s) = 1 + s^2$ with defining integrand $g(s) = 2s$ (by \Cref{thm:int-power}) and $f(u) = 1/(2u)$ in \Cref{thm:u-sub},
\[
\int_0^x \frac{t}{1 + t^2}\dd t = \int_1^{1 + x^2} \frac{\dd u}{2u} = \tfrac{1}{2}\log(1 + x^2).
\]
Substituting and rearranging gives the result.
\end{proof}

Other classical entries --- $\int x^n e^x\dd x$, $\int e^x \sin(x)\dd x$, the integral of $\arcsin$, and so on --- are similar showcases of \Cref{thm:uv-sub}, sometimes iterated, combined with our existing table.

\section{The Fundamental Theorem of Calculus}\label{sec:the Fundamental Theorem of Calculus}

We now cross the bridge from integral to differential calculus.

\begin{definition}
For $F\colon [a,b] \to \R$ and $x \in (a,b)$, the \emph{derivative} of $F$ at $x$ is
\[
\deriv{x}F(x) \coloneqq \lim_{h \to 0} \frac{F(x+h) - F(x)}{h},
\]
when the limit exists. We use one-sided versions at the endpoints $a$ and $b$.
\end{definition}

\begin{theorem}[Fundamental Theorem of Calculus, integral-to-differential]\label{thm:the Fundamental Theorem of Calculus-1}
Let $f$ be continuous on $[a,b]$, and define
\[
F(x) = \int_a^x f(t)\dd t.
\]
Then $F$ is differentiable on $[a,b]$ with $\deriv{x}F(x) = f(x)$.
\end{theorem}

\begin{proof}
Fix $x \in [a, b]$ and $\eps > 0$. By continuity of $f$, choose $\delta > 0$ so that $|t - x| < \delta$ implies $|f(t) - f(x)| < \eps$. For $0 < |h| < \delta$ with $x + h \in [a, b]$,
\[
\frac{F(x+h) - F(x)}{h} - f(x) = \frac{1}{h}\int_x^{x+h} \bigl(f(t) - f(x)\bigr)\dd t.
\]
The integrand is bounded by $\eps$ in absolute value on the interval of integration, which has length $|h|$. The right side therefore has absolute value at most $\eps$. Letting $h \to 0$, we deduce $\deriv{x}F(x) = f(x)$.
\end{proof}

For the converse direction we need one classical fact, which we state without proof.

\begin{lemma}[Constant criterion]\label{lem:zero-deriv}
If $H \colon [a,b] \to \R$ is differentiable on $[a,b]$ with $\deriv{x}H(x) \equiv 0$, then $H$ is constant.
\end{lemma}

\begin{corollary}[Fundamental Theorem of Calculus, differential-to-integral]\label{cor:the Fundamental Theorem of Calculus-2}
If $G \colon [a, b] \to \R$ is differentiable with $\deriv{x}G(x)$ continuous on $[a, b]$, then
\[
\int_a^b \deriv{t}G(t)\dd t = G(b) - G(a).
\]
\end{corollary}

\begin{proof}
Set $F(x) = \int_a^x \deriv{t}G(t)\dd t$. By \Cref{thm:the Fundamental Theorem of Calculus-1}, $\deriv{x}F(x) = \deriv{x}G(x)$ on $[a,b]$. Hence $\deriv{x}(F-G)(x) = 0$, and by \Cref{lem:zero-deriv}, $F - G$ is constant on $[a, b]$. Evaluating at $a$ gives $F(a) - G(a) = -G(a)$, so $F(x) = G(x) - G(a)$. Set $x = b$.
\end{proof}

\section{The differential calculus, harvested}\label{sec:harvest}

Every entry in our integral table now becomes an entry in a derivative table by \Cref{thm:the Fundamental Theorem of Calculus-1}. We have done all the work; we have only to read off the consequences.

\begin{theorem}[The standard derivative table]\label{thm:derivative-table}
The following derivative identities hold on the natural domains of the functions involved:
\begin{align*}
&\deriv{x}\exp(x) = \exp(x), && \deriv{x}\log(x) = 1/x, \\
&\deriv{x}x^a = a x^{a-1}\;(a \in \R), && \deriv{x}b^x = b^x \log(b), \\
&\deriv{x}\sin(x) = \cos(x), && \deriv{x}\cos(x) = -\sin(x), \\
&\deriv{x}\tan(x) = \sec^2(x), && \deriv{x}\cot(x) = -\csc^2(x), \\
&\deriv{x}\arctan(x) = \tfrac{1}{1+x^2}, && \deriv{x}\arcsin(x) = \tfrac{1}{\sqrt{1-x^2}}, \\
&\deriv{x}\sinh(x) = \cosh(x), && \deriv{x}\cosh(x) = \sinh(x), \\
&\deriv{x}\tanh(x) = \sech^2(x), && \deriv{x}\arsinh(x) = \tfrac{1}{\sqrt{1+x^2}}.
\end{align*}
\end{theorem}

\begin{proof}
Each is an instance of \Cref{thm:the Fundamental Theorem of Calculus-1} applied to the corresponding result of \Cref{sec:direct,sec:derived}. For example, by \Cref{thm:int-exp}, $\exp(x) = \exp(0) + \int_0^x \exp(t)\dd t = 1 + \int_0^x \exp(t)\dd t$. The integrand is continuous, so \Cref{thm:the Fundamental Theorem of Calculus-1} gives $\deriv{x}\exp(x) = \exp(x)$. By \Cref{def:log}, $\log(x) = \int_1^x dt/t$, so $\deriv{x}\log(x) = 1/x$. The remaining lines are equally immediate.
\end{proof}

The product rule and chain rule fall out of \Cref{thm:uv-sub,thm:u-sub} via \Cref{thm:the Fundamental Theorem of Calculus-1}.

\begin{corollary}[Product rule]\label{cor:product-rule}
If $u, v$ are differentiable on $[a, b]$ with continuous derivatives, then $uv$ is differentiable with $\deriv{x}(u(x)v(x)) = \deriv{x}u(x)\,v(x) + u(x)\,\deriv{x}v(x)$.
\end{corollary}

\begin{proof}
By \Cref{thm:uv-sub} applied to the integrals defining $u$ and $v$ from their derivatives, with $p(t) = \deriv{t}u(t)$ and $q(t) = \deriv{t}v(t)$ both continuous, for every $x \in [a, b]$,
\[
\int_a^x \deriv{t}u(t)\, v(t)\dd t + \int_a^x u(t)\,\deriv{t}v(t)\dd t = u(x)v(x) - u(a)v(a).
\]
The integrand on the left is continuous, so \Cref{thm:the Fundamental Theorem of Calculus-1} gives the left side as a differentiable function of $x$ with derivative $\deriv{x}u(x)\,v(x) + u(x)\,\deriv{x}v(x)$. The right side equals the left side and so is also differentiable; subtracting the constant $-u(a)v(a)$ shows that $x \mapsto u(x)v(x)$ is differentiable with $\deriv{x}(u(x)v(x)) = \deriv{x}u(x)\,v(x) + u(x)\,\deriv{x}v(x)$.
\end{proof}

\begin{corollary}[Chain rule]\label{cor:chain-rule}
Let $G\colon [a, b] \to \R$ be differentiable with continuous derivative $g$, and let $f$ be continuous on a closed interval containing $G([a, b])$. Let $F$ be any antiderivative of $f$ on that interval. Then $F \circ G$ is differentiable on $[a, b]$ with $\deriv{x}(F(G(x))) = f(G(x)) \cdot g(x)$.
\end{corollary}

\begin{proof}
By \Cref{thm:u-sub}, for every $x \in [a, b]$,
\[
\int_a^x f(G(t)) g(t)\dd t = \int_{G(a)}^{G(x)} f(u)\dd u = F(G(x)) - F(G(a)).
\]
The left side, as the integral of the continuous function $f \circ G \cdot g$, is a differentiable function of $x$ by \Cref{thm:the Fundamental Theorem of Calculus-1} with derivative $f(G(x))g(x)$. The right side equals the left side and so is also differentiable; subtracting the constant $-F(G(a))$ shows that $x \mapsto F(G(x))$ is differentiable with $\deriv{x}(F(G(x))) = f(G(x)) \cdot g(x)$.
\end{proof}

\begin{remark}
This is the structurally clean way to see what the chain rule and product rule are. They are exactly the differential shadows of the substitution and integration-by-parts theorems for integrals. In the differential-first alternative, this fact is obscured, because the derivative rules come first and the substitution rules are presented as their corollaries via the Fundamental Theorem of Calculus. the Fundamental Theorem of Calculus is the dictionary that translates between the two orderings, and either may be taken as the starting point.
\end{remark}

\section{Closing remarks}\label{sec:closing}

Having worked through the exposition, we may briefly recapitulate its virtues.

First, the direct integral evaluations are self-justifying. The geometric-series arguments for $\int b^t\dd t$ and the trigonometric integrals, the telescoping argument for $\int \sec^2(t)\dd t$, and the functional-equation derivation of $\log$ from $\int dt/t$ all reveal structural features of the integrands that are invisible in a derivative-first treatment. A student who has seen the same geometric-series engine drive three different integral evaluations in sequence comes away with a sharper sense of what makes those integrands tractable than one who has merely learned their antiderivatives by reversing the derivative table.

Second, stating and proving the substitution lemmas without reference to derivatives shows them in their natural form as theorems about Riemann sums. The chain rule and product rule recovered in \Cref{sec:harvest} are revealed as the differential shadows of the substitution and integration-by-parts lemmas. In the differential-first alternative, this structural relationship runs the other way and is therefore obscured.

Third, the Fundamental Theorem of Calculus appears once, at its climactic structural moment, in the direction integral-to-differential. It plays the role of a dictionary between two separately-built vocabularies rather than the role of a computational crutch wielded from the start.

A fair accounting of the disadvantages is also in order. The direct integral evaluations are not uniformly elegant; \Cref{thm:int-power}, for instance, requires a digression into Faulhaber's formula that a derivative-first treatment avoids by producing $\deriv{x}x^{n+1} = (n+1)x^n$ mechanically. The substitution $u = \sec(t) + \tan(t)$ that appears in \Cref{prop:int-sec} is a rabbit out of a hat in our framework, as it is in any framework, though a derivative-first treatment arguably mystifies it little more gracefully. And the computational fluency that differential-first students develop in their first semester --- the mechanical comfort with the chain rule and product rule --- is something our approach has to engineer separately, once the bridge has been crossed.

The differential-first tradition thus retains its appeal, particularly in curricula where early mechanical fluency is valued above structural transparency. It has its own distinguished adherents, and we have attempted neither to minimize its advantages nor to elide its costs. For instructors preparing an introductory course on calculus, however, we see no compelling reason to depart from the natural integral-to-differential order given here.

\bibliographystyle{amsplain}
\bibliography{integral-differential}

\end{document}